\documentclass[graybox]{article}

\usepackage{type1cm}        % activate if the above 3 fonts are
\usepackage{makeidx}         % allows index generation
\usepackage{graphicx}        % standard LaTeX graphics tool
\usepackage{subcaption}
\usepackage{multicol}        % used for the two-column index
\usepackage[bottom]{footmisc}% places footnotes at page bottom
\usepackage{authblk}

\usepackage{newtxtext}       % 

\usepackage{algorithm}
\usepackage{algpseudocode}
\usepackage{amsmath,amscd,amsthm,amssymb,mathrsfs,setspace}
\usepackage{mathtools}
\usepackage{comment}
\usepackage{mathtools}
\usepackage{xcolor}

\usepackage{hyperref}
\usepackage[nameinlink]{cleveref}
\usepackage{url}

\usepackage[varvw]{newtxmath}       % selects Times Roman as basic font

\newtheorem{theorem}{Theorem}[section]
\newtheorem{proposition}[theorem]{Proposition}
\newtheorem{lemma}[theorem]{Lemma}

\theoremstyle{remark}

\newcommand{\allone}{\mathbf{1}}
\newcommand{\N}{\mathbb{N}}
\newcommand{\R}{\mathbb{R}}
\newcommand{\st}{\operatorname{s.t.}}
\newcommand{\calT}{\mathcal{T}}
\newcommand{\MI}{\operatorname{MI}}

\newcommand{\ihat}{\hat \imath}
\newcommand{\jhat}{\hat \jmath}

\begin{document}

\title{Minimizing and Maximizing the Shannon Entropy for Fixed Marginals}

\author[1]{Paula Franke\footnote{\href{mailto:paula.franke@tu-dortmund.de}{paula.franke@tu-dortmund.de}, \href{https://orcid.org/0009-0009-6954-5854}{ORCID: 0009-0009-6954-5854}}}
\author[2]{Kay Hamacher\footnote{\href{mailto:kay.hamacher@tu-darmstadt.de}{kay.hamacher@tu-darmstadt.de}, \url{http://www.kay-hamacher.de}, \href{https://orcid.org/0000-0002-6921-8345}{ORCID: 0000-0002-6921-8345}}}
\author[1]{Paul Manns\footnote{\href{mailto:paul.manns@tu-dortmund.de}{paul.manns@tu-dortmund.de}, \href{https://orcid.org/0000-0003-0654-6613}{ORCID: https://orcid.org/0000-0002-6921-8345}}}
\affil[1]{TU Dortmund, 44227 Dortmund, Germany}
\affil[2]{TU Darmstadt, Schnittspahnstr. 2, 64287 Darmstadt, Germany}
\date{}

\setcounter{Maxaffil}{0}
\renewcommand\Affilfont{\itshape\small}

\maketitle

\abstract{The mutual information (MI) between two random variables is an important correlation measure in data analysis. The Shannon entropy of a joint probability distribution is the variable part under fixed marginals. We aim to minimize and maximize it to obtain the largest and smallest MI possible in this case, leading to a scaled MI ratio for better comparability. We present algorithmic approaches and optimal solutions for a set of problem instances based on data from molecular evolution. We show that this allows us to construct a sensible, systematic correction to raw  MI values.}%, improving the{\color{pink}{signal-to-noise ratio}}.}

\section{Introduction}\label{sec:introduction}
We seek to both minimize and maximize the Shannon entropy of a joint, discrete probability
distribution $P \in \R^{n \times m}$, $0 \le P$,
$\allone^\top P \allone = 1$ that satisfies the two marginal
distributions $\mu \in \R^n$, and $\nu \in \R^m$. The Shannon entropy of $P$ reads
\begin{gather}\label{eq:objective}
H(P) \coloneqq \sum_{i=1}^{n}\sum_{j=1}^m h(P_{ij}) = - \sum_{i=1}^{n}\sum_{j=1}^m P_{ij} \log (P_{ij})
\end{gather}
with the convention $0 \log (0) \coloneqq 0$ so that $h : [0,\infty) \ni x \mapsto - x \log(x) \in \R$ is
continuous.

We are interested in the minimization and maximization of $H(P)$ because it is the variable 
part (under fixed marginal distributions) in the mutual information $\MI(P) = H(\mu)+H(\nu)-H(P)$
between two random variables, where $H(\mu) = \sum_{i=1}^n h(\mu_i)$, $H(\nu) = \sum_{j=1}^m h(\nu_j)$ with a slight abuse of notation. The $\MI$ is a generalized correlation measure
and can easily be shown to be bounded from above by $\max\{H(\mu),H(\nu)\}$. Now, if we want
to compare the degree of generalized correlation between different random variables charactarized
by $\mu,\nu,P$ and $\mu',\nu',P'$ a direct comparison between MI and MI' is not sensible as the overall
scale is given by $\max\{H(\mu),H(\nu)\}$ and $\max\{H(\mu'),H(\nu')\}$, respectively.
Previous approaches \cite{Hamacher2013_3,Hamacher2010_5} to cope with this problem relied on
resampling techniques
to create {\em in silico} distributions of MI values for which one then can compute a $Z$-score.
This, however, is only an approximation to the full co-domain of possible MI values.
An important application lies in molecular biology where finding generalized correlations in
evolving molecular sequences contributes to the understanding of the selective pressure
on the functionality of those biomolecules \cite{Hamacher2008}.
Consequently, we strive for minimizing and maximizing $\MI$ to be able to determine a standardization ratio based on the actually
achievable range of $\MI$-values for comparison.
Note that minimizing or maximizing $\MI$ for fixed $\mu$, $\nu$ is equivalent to maximizing or minimizing $H$ so that this is the object
of interest in this study.

In order to assess the optimization techniques
we propose below, we use the empirical joint distribution on the evolution of the HIV1-protease---an important
target for drugs targeting the progression of HIV infection. The data was taken from
the HIV positive selection mutation database \cite{Pan2007,Chen2004}.
Details of the data preparation are to be found in \cite{Hamacher2008}.

Since $H$ is strictly concave, its minimization is the more difficult task, which we do by solving a successively 
refined family of mixed-integer linear problems (MILPs) that approximate the limit problem from below.
We have implemented this procedure using the off-the-shelf solver Gurobi \cite{gurobi} to model and solve
the MILP instances of our benchmark. For the maximization, we argue that we may restrict to a polytope that is a 
subset of the feasible set, on which $H$ is differentiable and use the open source interior point solver Ipopt 
\cite{wachter2006implementation} to solve the instances of our benchmark.

After introducing our notation and standing assumptions below, we analyze the minimization
of $H$ over the set of joint probability distributions that satisfy the marginals $\mu$ and $\nu$ in \cref{sec:minimization}
and the maximization in \cref{sec:maximization}. In these sections, we also make the case for our respective algorithmic approaches.
We provide information on our computational setup and benchmark as well as the achieved results in \cref{sec:numerics}, and draw a brief conclusion in \cref{sec:conclusion}.

\subsection*{Notation and Standing Assumptions} For $k \in \N$, $\allone_k \in \R^k$ denotes the vector containing $1$ in
all of its entries. Throughout \cref{sec:minimization,sec:maximization}, we assume that $\mu \in \R^n$ and $\nu \in \R^m$ are two arbitrary but 
fixed discrete probability distributions with $m$, $n \in \N$, that is, we have $\mu \ge 0$, $\allone_n^\top \mu = 1$ and $\nu \ge 0$,
$\allone_m^\top \nu = 1$. Note that $\ge$ is understood componentwise for vectors and matrices. Without loss of generality,
we assume that $\mu$, $\nu$ only contain strictly positive entries (every joint probability distribution $P \in \R^{n\times m}$
satisfying the marginal $P\allone_m = \mu$ with $\mu_i = 0$ for some $i$ satisfies $P_{ij} = 0$ for all $j$ and $h(P_{ij}) = 0$
and thus does not influence $H(P)$).

\section{Minimizing $H$ for Fixed Marginals}\label{sec:minimization}
The entropy minimization problem reads
\begin{gather}\label{eq:Emin}
\min_{P}\enskip H(P)
\enskip\st\enskip P \in \R^{n\times m}, 0 \le P, P\allone_m = \mu, P^\top\allone_n = \nu.\tag{E$\min$}
\end{gather}
We state the following facts that follow from 
basic optimization theory without proof.
\begin{proposition}\label{sec:Emin_basics}
\eqref{eq:Emin} is a strictly concave minimization problem that has a compact feasible set
and admits a global minimizer. At least one global minimizer is at an extreme point of the polytope
$\{P \in \R^{n\times m} : 0 \le P, P\allone_m = \mu, P^\top\allone_n = \nu\}$.
\end{proposition}
We stress that we must solve \eqref{eq:Emin} to (approximate) global optimality in our application.
Since concave minimization problems are NP-hard in general \cite{murty1987some,pardalos1991quadratic}
and we are not aware of any efficient algorithm to solve \eqref{eq:Emin}, we opt for a solution with a branch-\&-bound-based MILP solver. Specifically, we recall that
our objective $H$ is separable, that is, it decomposes into a sum of one-dimensional functions;
see \eqref{eq:objective}. Piecewise-linear functions can be modeled in MILPs using
binary variables that indicate the interval of the evaluation point using so-called SOS-2 constraints; see \cite{geissler2011using}.
This technique is already built into the API of Gurobi.
We thus propose to alternately solve
\eqref{eq:Emin_k}, where $H$ is replaced by a piecewise
linear surrogate, and update the surrogate based on the computed solution.
Let $ l \in \N$ and $\calT = \{\tau_1, \dots, \tau_l\} \subset [0,1]$ be a finite set of support points for $h$ on $[0,1]$. Then $h^{\calT}$
defined as
\begin{gather}\label{eq:1dsurrogate}
h^\calT (x) \coloneq h(\tau_i) + \frac{h(\tau_{i+1})-h(\tau_i)}{\tau_{i+1}-\tau_i}(x-\tau_i) \quad \text{for }x \in [\tau_i, \tau_{i+1}]
\end{gather}
is a piecewise linear function that underestimates $h$ on $[0,1]$. Consequently, updating
the surrogate model means adding additional support points to $\calT_{ij}$ for one of the summands
in the sum in $H$. This gives the surrogate problem:
\begin{gather}\label{eq:Emin_k}
\min_{P}\enskip H_k(P)
\enskip\st\enskip P \in \R^{n\times m}, 0 \le P, P\allone_m = \mu, P^\top\allone_n = \nu.\tag{E$\min$(k)}
\end{gather}
with $H_k(P) = \sum_{i=1}^{n}\sum_{j=1}^m h^{{\calT}_{ij}}(P_{ij})$ after $k$ alternations between solving and
updating the sets $\calT_{ij}$. Note that \eqref{eq:Emin_k} can be implemented as an MILP using the aforementioned modeling for the surrogate \eqref{eq:1dsurrogate} and the fact that
its remaining constraint set is a polytope. 

\begin{algorithm}
\caption{Alternation Algorithm for
	Computing $\varepsilon$-optimal Solution
	to \eqref{eq:Emin}}\label{alg:alternation}
\begin{algorithmic}[1] 
\Require $\varepsilon > 0$
\State $k \gets 0$, $\varepsilon_0 \gets \infty$
\State $\calT_{ij} \gets \{0, \min\{\mu_i, \nu_j\}, 1\}$ for $i=1,\dots, n$, $j=1,\dots,m$
\While{$\varepsilon_k \ge \varepsilon$}
\State $k \gets k + 1$
\State $P^{\min}_k \gets $ global minimizer to \eqref{eq:Emin_k} \label{ln:milp_solve}
\State $\ihat, \jhat \gets \arg\max_{ij} \enspace \lvert h((P^{\min}_k)_{ij}) - h^{\calT_{ij}}((P^{\min}_k)_{ij})\rvert$ \label{ln:maxerror}
\State $\calT_{\ihat\jhat} \gets \calT_{\ihat\jhat} \cup \{(P^{\min}_k)_{\ihat\jhat}\}$ \label{ln:update}
\State $\varepsilon_{k} \gets \left\lvert \frac{H(P^{\min}_k) - H_k(P_k^{\min})}{H_k(P_k^{\min})} \right\rvert$
\EndWhile
\end{algorithmic}
\end{algorithm}

The alternation scheme is
formalized in \cref{alg:alternation}.
Our update strategy for the surrogate models Line \ref{ln:update} implies that
\cref{alg:alternation} closes the gap between $H(P^{\min})$ and $H_k(P^{\min}_k)$, where
$P^{\min}$ is any global minimizer of \eqref{eq:Emin} and $P^{\min}_k$ is the global minimizer
of \eqref{eq:Emin_k} computed in
Line \ref{ln:milp_solve}.

\begin{theorem}
Let $\varepsilon > 0$. Then \cref{alg:alternation} terminates after finitely many iterations with
$|H(P^{\min}) - H_{k_{\max}}(P^{\min}_{k_{\max}})| < \varepsilon$ if $k_{\max}$ denotes the final iteration.
\end{theorem}
\begin{proof}[Sketch of Proof]
By a compactness and subsequence argument, we 
restrict to a converging
 subsequence $p^{k_l} \coloneq (P^{\min}_{k_l})_{\ihat\jhat}$ for fixed $(\ihat,\jhat)$ (cf.\ Line \ref{ln:maxerror}) with limit $\Bar{p}\in [0,1]$.

For each $k \in \N$, let $L^k, U^k \in \calT_{\ihat\jhat}$ be the support points in the $k^{th}$ iteration such that $p^k \in [L^k, U^k]$. As $p^k$ is added as a support point in each iteration, either $p^k-L^k \to 0$ or $U^k - p^k \to 0$ holds. Assume that the former holds; the latter case follows analogously. Recall from equation \eqref{eq:1dsurrogate} that
\begin{gather}
    h^{\calT_{\ihat\jhat}}(p^k) = h^{\calT_{\ihat\jhat}} (L^k) + \frac{h^{\calT_{\ihat\jhat}}(U^k)-h^{\calT_{\ihat\jhat}}(L^k)}{U^k-L^k}(p^k -L^k). \label{eq:interpolation}
\end{gather}

Now, if $\Bar{p}>0$, then the difference quotient in \eqref{eq:interpolation} is uniformly bounded. As $p^k-L^k \to 0$ and $p^k \to \Bar{p}$, it holds that $h^{\calT_{\ihat\jhat}}(p^k) \to h(\Bar{p})$. Similarly, if $\Bar{p}=0$, bounding $(p^k-L^k)$ from above with $(U^k-L^k)$ together with $U^k - L^k \to 0$ and continuity of $h$ gives $\lim_k h^{\calT_{\ihat\jhat}}(p^k) \leq h(0) =0$. This yields $h^{\calT_{\ihat\jhat}}(p^k) \to h(\Bar{p})$ as $h^{\calT_{\ihat\jhat}} \geq 0$.

Lastly, note that $h^{\calT_{\ihat\jhat}}(p^k) \leq h(P^{\min}_{\ihat\jhat}) \leq  h(\Bar{p})$ and therefore, $\Bar{p}$ is a minimizer of $h$. As $(\ihat,\jhat)$ is the index of maximum error, the same holds for all indices.
\end{proof}

\section{Maximizing $H$ for Fixed Marginals}\label{sec:maximization}
The entropy maximization problem reads
\begin{gather}\label{eq:Emax}
\max_{P}\enskip H(P)
\enskip\st\enskip 0 \le P, P\allone_n = \nu, P^\top\allone_m = \mu.\tag{E$\max$}
\end{gather}
We state the following facts that follow from 
basic optimization theory without proof.
\begin{proposition}\label{prp:Emax_basics}
\eqref{eq:Emax} is a strictly concave maximization problem that satisfies Slater's condition,
has a unique KKT-point that is a global maximizer and satisfies strong duality.
\end{proposition}
In light of \cref{prp:Emax_basics}, it is sensible 
to solve \eqref{eq:Emin} numerically with a 
gradient-based interior point solver. Because
$H$ is not differentiable at $P$ if $P_{ij} = 0$ holds for some $i$, $j$, we recall from optimal transport theory that optimal solutions to \eqref{eq:Emax} are bounded away from zero, allowing us to add positive lower bounds on $P$ to \eqref{eq:Emax}.
\begin{lemma}[Proposition 3.4 in \cite{clason2021entropic}]\label{lem:Pmax}
Let $\mu_i > 0$ for all $i \in \{1,\ldots,n\}$, $\nu_j > 0$ for all $j \in \{1,\ldots,m\}$.
Let $P^{\max}$ solve \eqref{eq:Emax}. Then there exists an $\eta >0$ such that $P^{\max}_{ij} > \eta$ for all $i \in \{1,\ldots,n\}$ and all $j \in \{1,\ldots,m\}$.
\end{lemma}

\section{Computational Setup and Results}\label{sec:numerics}

The Alternation Algorithm \ref{alg:alternation} was implemented in Python. The surrogate problems \eqref{eq:Emin_k} in ln.\ \ref{ln:milp_solve} were solved with the off-the-shelf solver Gurobi using the built-in piecewise linear model. The program was validated by successfully solving non-trivial instances with known solutions. The maximization was implemented using the open-source solver Ipopt. The program was validated with the built-in first order derivative check using finite differences. Results were validated by checking the scaled NLP error and the condition number of the reduced Hessian of the Lagrangian at the optimal point.

We have solved all of the 4851 benchmark instances. Both programs were executed on a Lenovo ThinkPad E14 Gen 6 with AMD Ryzen 7 processor.

The minimization of the entropy function $H$ with Algorithm \ref{alg:alternation} was successful for a prescribed tolerance of
$\varepsilon = 10^{-4}$. Some instances showed exceptionally high running times of up to 5 hours and 14 minutes. For $\varepsilon=10^{-3}$, in comparison, the longest running time for a single instance was around 10 minutes. Overall, the program provided optimal solutions for all instances and can thus be deemed a success.

The Ipopt-based maximization resulted in optimal solutions for all instances when using the HSL solver ma57 \cite{duff2004ma57} and providing Hessian information. The default solver MUMPS did not find optimal solutions for 112 instances. Therefore, only the run with ma57 is considered. Validation efforts for the obtained results showed that condition numbers of the reduced Hessians ranged from 1 to 12,799,731 with a median of 180,350.

From the optimal solutions $P^{\min}_{k\ell}$ and $P^{\max}_{k\ell}$ for all problem instances $(k,\ell) \in \{1,\dots,99\}^2$, we determined the scaled MI ratio
$\frac{MI(M_{k\ell})-MI(P^{\min})}{MI(P^{\max})-MI(P^{\min})}$
where $M_{k\ell}$ is the 
contingency table of amino acid occurrences for 
positions $k$ and $\ell$ in the molecule. The one-point entropies cancel each other and 
the scaled MI ratio reads
\begin{gather}
    \rho(M_{k\ell})=\frac{H(M_{k\ell})-H(P^{\max})}{H(P^{\min})-H(P^{\max})}
\end{gather}
per data point (problem instance), which is visualized in
Figure \ref{fig:quotient_visualization}. Figure \ref{fig:scatter_entropy} shows the correlation between the entropy values $H(M_{k\ell})$ and the entropy quotient $\rho(M_{k\ell})$ for non-trivial\footnote{We consider contingency tables with just one row/column as trivial because the induced feasible sets are singletons.} data points, scaled logarithmically ([nat]). The generalized correlation (Kendall's $\tau$) is $0.22$  % 0.2174332 
for all cases and
$0.67$ % 0.6670904 
for the non-trivial cases.\footnote{Both tests showed a $p$-value $< 2.2\cdot 10^{-16}$
under the null hypothesis of no correlation
($\tau=0$) and were
computed via the internal routine of 
the statistics package {\tt R} \cite{R}.}

As is evident from Fig.~\ref{fig:combined_quotient} the normalization procedure by computing $\rho$s gives rise to ($\tau \approx 0.6$) similar values, but also induces a systematic correction (Fig.~\ref{fig:scatter_entropy}). At the same time, the results in Fig.~\ref{fig:quotient_visualization} show a more nuanced picture than the pure MI which tends to result in \emph{stripes} of high or low MI values for some $k$ or $\ell$. Here,
we omitted the trivial entries at $k,\ell \in \{1,26,27,40,44,51,80,86\}$ to obtain \emph{effective ids} and replaced the diagonal entries with an average to exploit the full range of intensities. A detailed analysis of the biomolecular/evolutionary implications will be given in a forthcoming publication \cite{future}.
\begin{figure}
    \centering
    \begin{subfigure}[b]{0.45\textwidth}
        \centering
        \includegraphics[width=.85\textwidth]{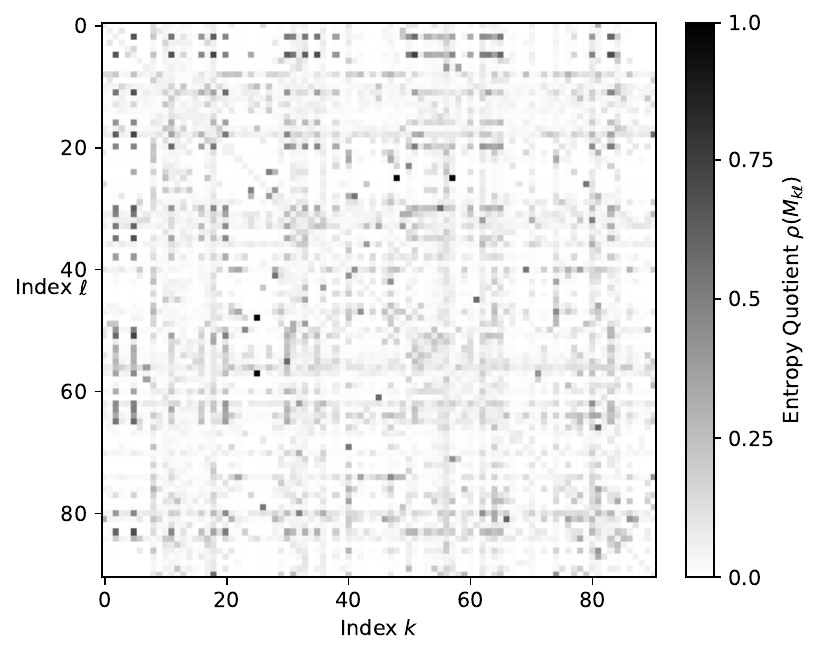}
        \caption{Visualization of $\rho(M_{k\ell})$ for non-trivial cases.}
        \label{fig:quotient_visualization}
    \end{subfigure}
    \hfill
    \begin{subfigure}[b]{0.45\textwidth}
        \centering
        \includegraphics[width=\textwidth]{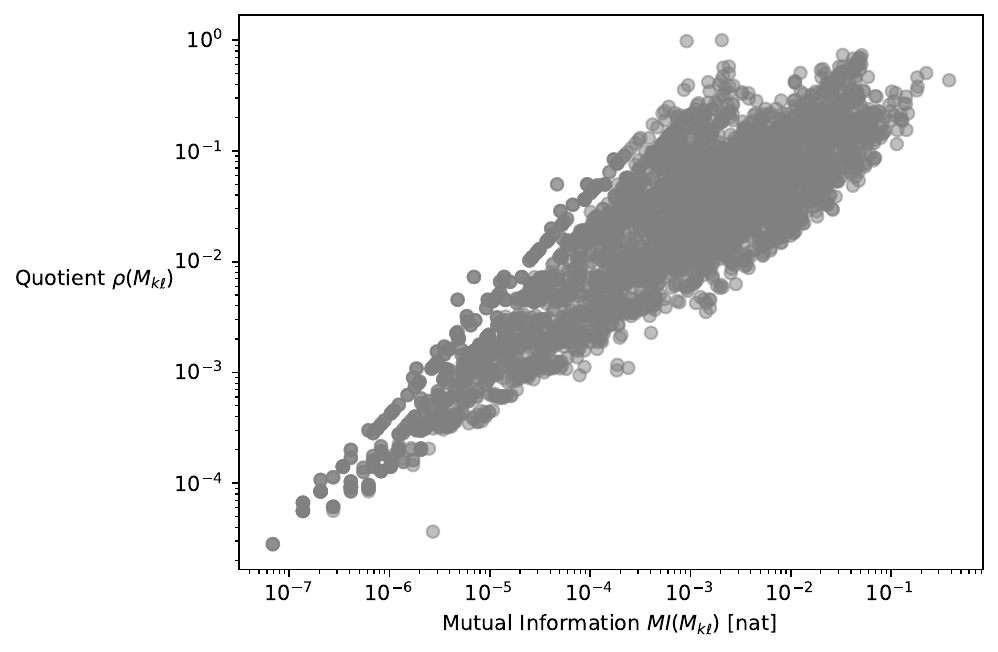}
        \caption{Scatterplot $MI(M_{k\ell})$ vs. $\rho(M_{k\ell})$ for non-trivial cases.}
        \label{fig:scatter_entropy}
    \end{subfigure}
    \caption{Visualizations of the entropy quotient $\rho$ and comparison with mutual information.}
    \label{fig:combined_quotient}
\end{figure}

\section{Conclusion}\label{sec:conclusion}
The minimization and maximization algorithms proposed in \S\ref{sec:minimization} and \S\ref{sec:maximization} provided valuable results for all data points that we applied to find the scaled MI ratio in \S\ref{sec:numerics}. Future work seems promising for the underlying biological research, as a significant generalized correlation between the MI values and the scaled MI ratios was observed.

{\small \textbf{Competing Interests}\quad The authors have no conflicts of interest to declare that are relevant to the content of this article.}

\bibliographystyle{plain}
\bibliography{kh}

\begin{thebibliography}{10}
\providecommand{\url}[1]{{#1}}
\providecommand{\urlprefix}{URL }
\expandafter\ifx\csname urlstyle\endcsname\relax
  \providecommand{\doi}[1]{DOI \discretionary{}{}{}#1}\else
  \providecommand{\doi}{DOI \discretionary{}{}{}\begingroup
  \urlstyle{rm}\Url}\fi

\bibitem{Hamacher2013_3}
M.~Waechter, K.~Jaeger, D.~Thuerck, S.~Weissgraeber, S.~Widmer, M.~Goesele,
  K.~Hamacher, Concurrency and Computation: Practice and Experience
  \textbf{26}(6), 1278 (2014).
\newblock \doi{10.1002/cpe.3074}.
\newblock \urlprefix\url{http://dx.doi.org/10.1002/cpe.3074}

\bibitem{Hamacher2010_5}
S.~Bremm, T.~Schreck, P.~Boba, S.~Held, K.~Hamacher, BMC Bioinformatics
  \textbf{11}, 330 (2010)

\bibitem{Hamacher2008}
K.~Hamacher, Gene \textbf{422}, 30 (2008)

\bibitem{Pan2007}
C.~Pan, J.~Kim, L.~Chen, Q.~Wang, C.~Lee, Nuc. Acids Res. \textbf{35}, D371
  (2007)

\bibitem{Chen2004}
L.~Chen, A.~Perlina, C.J. Lee, J. Virol. \textbf{78}(7), 3722 (2004).
\newblock \doi{10.1128/JVI.78.7.3722-3732.2004}

\bibitem{gurobi}
{Gurobi Optimization, LLC}.
\newblock {Gurobi Optimizer Reference Manual} (2025).
\newblock \urlprefix\url{https://www.gurobi.com}

\bibitem{wachter2006implementation}
A.~W{\"a}chter, L.T. Biegler, Mathematical Programming \textbf{106}(1), 25
  (2006)

\bibitem{murty1987some}
K.G. Murty, S.N. Kabadi, Mathematical Programming \textbf{37}, 117 (1987)

\bibitem{pardalos1991quadratic}
P.M. Pardalos, S.A. Vavasis, Journal of Global Optimization \textbf{1}(1), 15
  (1991)

\bibitem{geissler2011using}
B.~Gei{\ss}ler, A.~Martin, A.~Morsi, L.~Schewe, in \emph{Mixed Integer
  Nonlinear Programming} (Springer, 2011), pp. 287--314

\bibitem{clason2021entropic}
C.~Clason, D.A. Lorenz, H.~Mahler, B.~Wirth, Journal of Mathematical Analysis
  and Applications \textbf{494}(1), 124432 (2021)

\bibitem{duff2004ma57}
I.S. Duff, ACM Transactions on Mathematical Software (TOMS) \textbf{30}(2), 118
  (2004)

\bibitem{R}
{R Core Team}, \emph{R: A Language and Environment for Statistical Computing}.
\newblock R Foundation for Statistical Computing, Vienna, Austria (2022).
\newblock \urlprefix\url{https://www.R-project.org/}

\bibitem{future}
Franke, Hamacher, Manns, et~al.,  To be published

\end{thebibliography}

\end{document}